\documentclass[12pt,reqno,a4paper]{amsart}
\usepackage{amsmath,amssymb,amsthm,amsfonts}
\usepackage{mathrsfs}
\usepackage{xcolor}
\usepackage{graphicx}
\usepackage{soul}

\newtheorem{theorem}{Theorem}
\newtheorem{lemma}{Lemma}

\theoremstyle{definition}
\newtheorem{defn}{Definition}
\theoremstyle{remark}
\newtheorem{remark}{Remark}

\usepackage[ left=1.5in,
top=0.6in,
left=1.2in,
right=1.2in,
bottom=0.8in,
headsep=0.25in,
headheight=20pt,
heightrounded,
a4paper]
{geometry}

\newcommand{\Rr}{\mathbb{R}}
\newcommand{\Nn}{\mathbb{N}}
\newcommand{\Zz}{\mathbb{Z}}

\begin{document}

\title[Multidimensional contracted rotations]
{Multidimensional contracted rotations}

\author[Gaiv\~ao]{Jos\'e Pedro Gaiv\~ao}
\thanks{J. P. Gaiv\~ao was partially supported by the Project CEMAPRE/REM - UIDB/05069/2020 - financed by FCT/MCTES through national funds.}
\address{ISEG/UL - Universidade de Lisboa, Department of Mathematics; REM - Research in Economics and Mathematics, CEMAPRE\\
Rua do Quelhas 6, 1200-781 Lisboa, Portugal}
\email{jpgaivao@iseg.ulisboa.pt}

\author[Pires]{Benito Pires}\thanks{
B. Pires was partially supported 
by the S\~ao Paulo Research Foundation (FAPESP), Brasil. Process Numbers 2019/10269-3,  2022/14130-2, 2024/15612-6 and 2025/14256-4}.
\address{Departamento de Computa\c c\~ao e Matem\'atica\\ Faculdade de Filosofia, Ci\^encias e Letras, Universidade de S\~ao Paulo\\ Ribeir\~ao Preto, SP, 14040-901, Brazil}
\email{benito@usp.br}

\date{\today}  

\begin{abstract}
We study the dynamics of multidimensional contracted rotations and address a problem posed by Y. Bugeaud and J-P. Conze in \textit{Acta Arithmetica} in 1999. More precisely, we show that if $A$ is an invertible linear contraction of $\mathbb{R}^d$, then the map $f: [0,1)^d\to [0,1)^d$ defined by  $f(x) = Ax +b\,\,(\textrm{mod}\,\mathbb{Z}^d)$ is asymptotically periodic for Lebesgue almost all $b\in\mathbb{R}^d$. We also include an example of a family of multidimensional contracted rotations $(d>1)$ not conjugate to the  product of one-dimensional contracted rotations $(d=1)$, showing that our result cannot be reduced to or derived from the one-dimensional result of Bugeaud and Conze.
\end{abstract}

\maketitle
\section{Introduction and Main Result}

Given $d\in\Nn$, the multidimensional contracted rotation is the map {$f\colon [0,1)^d\to[0,1)^d$} which assigns to every $x\in[0,1)^d$ the vector 
\begin{equation}\label{mapf1}
f(x):=Ax+b\pmod{\Zz^d},
\end{equation}
where $A\in\mathrm{GL}_d$ (i.e. $A$ is an invertible $d\times d$ matrix),  $\|A\|<1$ and $b\in\Rr^d$. Here $\Vert A\Vert = \sup_{x\neq 0} \Vert Ax\Vert\big/\Vert x\Vert$ and $\Vert x\Vert=\max_{1\le j\le d} \vert x_j\vert$ {for all $x=(x_j)$.
 Whenever we need to stress the dependence of $f$ on $A$ and $b$, we shall write $f_{A,b}$ to denote the multidimensional contracted rotation. 
The \textit{$\omega$-limit set} of $f$ is defined by
\begin{equation}\label{omegalimitset}
\omega(f) := \bigcup_{x\in [0,1)^d}  \omega(f,x), \quad\textrm{where}\quad
 \omega(f,x) := \bigcap_{m\ge 1} \overline{\bigcup_{n\ge m} {\left\{f^n(x)\right\}}}
\end{equation}
denotes the $\omega$-limit set of a point $x\in [0,1)^d$ under $f$.
We say that $f$ is \textit{asymptotically periodic} if $\omega(f)$ is the union of finitely many periodic orbits of $f$.

In this article, we address a problem stated by Y. Bugeaud and J-P. Conze in \cite[p. 218]{zbMATH01300448}. More precisely, we prove the following result.

\begin{theorem}\label{thm1} Let $A\in\mathrm{GL}_d$ with $\|A\|<1$. Then
 for Lebesgue almost every $b\in\mathbb{R}^d$, the map $f$ defined by \eqref{mapf1} is asymptotically periodic. 
\end{theorem}



{Theorem \ref{thm1} is a statement about the set $$
\mathcal{E}_A:=\{b\in [0,1)^d\colon f_{A,b}\text{ is not asymptotically periodic}\}
$$
called the exceptional set of the family ${f_{A,b},\,b\in [0,1)^d}$. It states that $\mathrm{Leb}\big(\mathcal{E}_A\big)=0$, where $\mathrm{Leb}$ denotes the $d$-dimensional Lebesgue measure.
 
In the one-dimensional case ($d=1$), substantially stronger results are known, see for instance \cite{B93, B04, zbMATH01300448, MR1861988, BKLN21, Arnaud, LN18}. More precisely, circle contracted rotations $f_{\lambda,b}:x\in [0,1)\mapsto \lambda x+b\pmod{1}$ have strictly increasing lifts which can be used to define a rotation number $\rho(f_{\lambda,b})\in[0,1)$,  varying continuously in $b$,  characterizing the average asymptotic rotation of every orbit of $f_{\lambda,b}$. A circle contracted rotation $f_{\lambda,b}$ is asymptotically periodic if and only if the rotation number is a rational number. Hence, the exceptional set of circle contracted rotations is 
$$
\mathcal{E}_{\lambda} := \{b\in[0,1)\colon \rho(f_{\lambda,b})\notin \mathbb{Q}\}.
$$
From the continuity of $b\mapsto \rho(f_{\lambda,b})$ and the fact that $\rho(f_{\lambda,0})=0$ and \linebreak $\lim_{b\to 1^-} \rho(f_{\lambda,b})= 1$,  we know that the exceptional set $\mathcal{E}_{\lambda}$ is uncountable.  Moreover, it can be shown that $\mathcal{E}_{\lambda}$ has zero Lebesgue measure, which is an easy consequence of the explicit decomposition (see \cite[Theorem 1]{MR1861988}):
$$
\mathcal{E}_{\lambda} = [0,1)\bigg\backslash \bigcup_{1\leq p<q,  (p,q)=1} \left[\frac{1-\lambda}{1-\lambda^q}S(\lambda,p/q), \frac{1-\lambda}{1-\lambda^q}(S(\lambda,p/q) + \lambda^{q-1}-\lambda^q)\right],
$$
where $S(\lambda,p/q)= 1+ \sum_{k=1}^{q-2}([(k+1)p/q]-[kp/q])\lambda^k$. In fact, $\mathcal{E}_{\lambda}$ has zero Hausdorff dimension \cite[Theorem 5]{LN18}. See Figure \ref{fig2}, in the Appendix~\ref{appendix figure} at the end of this article, illustrating the relation between  the rotation number of $f_{\lambda,b}$ and the parameters $\lambda$ and $b$.  

As an immediate consequence of the theory of circle contracted rotations, it follows that whenever $A=\operatorname{diag}(\lambda_1,\ldots,\lambda_d)$
is a contracting diagonal matrix with $\lambda_1,\ldots,\lambda_d\in(0,1)$, the exceptional set $\mathcal E_A$ is uncountable and has Hausdorff dimension $d-1$ (in particular has zero Lebesgue measure). Indeed, in this case we have $f_{A,b}=\prod_{i=1}^d f_{\lambda_i,b_i}$, i.e., $f_{A,b}$ is the product of  circle contracted rotations $f_{\lambda_i,b_i}:x\in [0,1)\mapsto \lambda_i x+b_i\pmod{1}$. Therefore,  $\mathcal{E}_A =\bigcup_{i=1}^d p_i^{-1}\big(\mathcal{E}_{\lambda_i}\big)$,  where $\mathcal{E}_{\lambda_i}$ is the exceptional set of the circle contracted rotation $f_{\lambda_i,b_i}$ and $p_i$ is the projection $(u_1,\ldots,u_d)\mapsto u_i$.

It is unclear whether the rotation number approach can be extended to higher dimensions so as to provide a similar description of the exceptional set $\mathcal E_A$ as in the one-dimensional case. Beyond the special case of diagonal matrices, the following result shows that, in the non-diagonalizable case, the study of $\mathcal E_A$ cannot be reduced to one-dimensional dynamics.

\begin{theorem}\label{theoremtwo} There exist a non-diagonalizable matrix $A\in\mathrm{GL}_2$ with $\|A\|<1$ and a Lebesgue full set $U\subset\mathbb{R}^2$ such that for each $b\in U$,  the map $f=f_{A,b}$ defined by \eqref{mapf1} is not $C^1$ conjugate to the  product of one-dimensional contracted rotations. 
\end{theorem}

Apart from the particular cases mentioned above, to the best of our knowledge, no general results on multidimensional contracted rotations were known prior to this work. Progress on this problem has only become possible through recent advances in the theory of multidimensional piecewise contractions \cite{GaivaoPires2024,JAIN_LIVERANI_2025}. Although several results on multidimensional piecewise contractions are available (see \cite{BD08,CGMU16,zbMATH04182365}), they do not apply to the parametrized families arising in multidimensional contracted rotations.

A key difficulty in the proof of Theorem \ref{thm1} is that, when the $b$-family associated with \eqref{mapf1} is rewritten as a family of maps on $[0,1)^d$, the parameter $b$ appears not only in the definition of the map but also in its domains of continuity. Consequently, the resulting family falls outside the scope of the existing theory. Another difficulty is that Theorem \ref{thm1} requires asymptotic periodicity on the whole domain, whereas available results establish such conclusions only on the set of regular points. This stronger statement is provided by Theorem \ref{thm:Gmu2}.

\section{Proofs}

\subsection{Preliminaries}
Throughout this section, fix $A={(a_{j\ell})}\in\mathrm{GL}_d$ with   {$\|A\|=\max_j\sum_{\ell} |a_{j\ell}|<1$}. We will consider $b\in\mathbb{R}^d$ as a parameter. 
In this way, the map $f$ in \eqref{mapf1}  becomes the \mbox{$b$-parameter} family of maps
$f_b: [0,1)^d \to [0,1)^d$, $b\in\mathbb{R}^d$, defined by 
\begin{equation}\label{mapf2}
f_b(x):=Ax+b\pmod{\Zz^d}.
\end{equation}
Denoting {by} $\lfloor\cdot\rfloor$ the component-wise floor function and {by} $\{x\}=x-\lfloor x\rfloor$, $x\in\Rr^d$, the component-wise fractional part, then \eqref{mapf2} reads
\begin{equation}\label{f}
f_b(x)=\{Ax+b\}=Ax+b-\lfloor Ax+b\rfloor,\quad \forall x\in[0,1)^d.
\end{equation}

\begin{figure}[ht]
    \centering
    \begin{minipage}{0.49\textwidth}
        \centering
        \includegraphics[width=\linewidth]{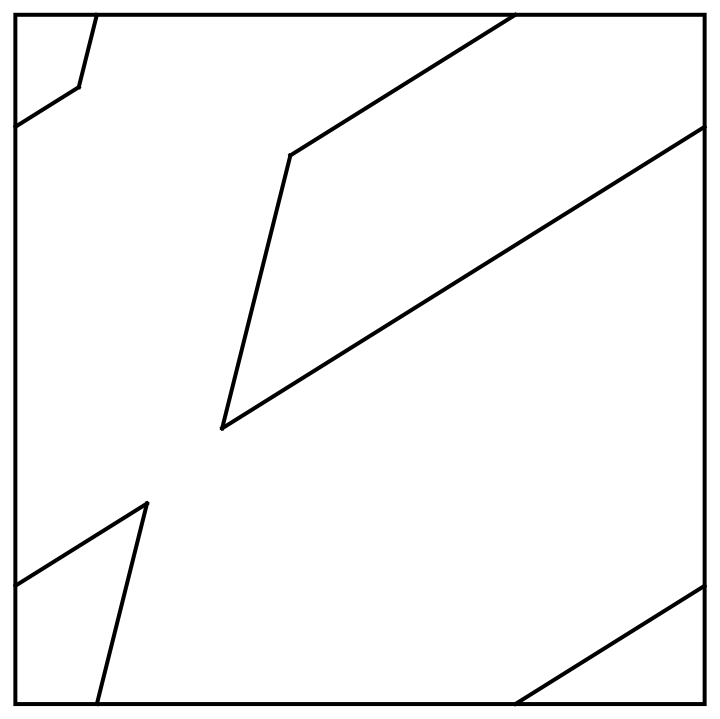}
    \end{minipage}\hfill
    \begin{minipage}{0.49\textwidth}
        \centering
        \includegraphics[width=\linewidth]{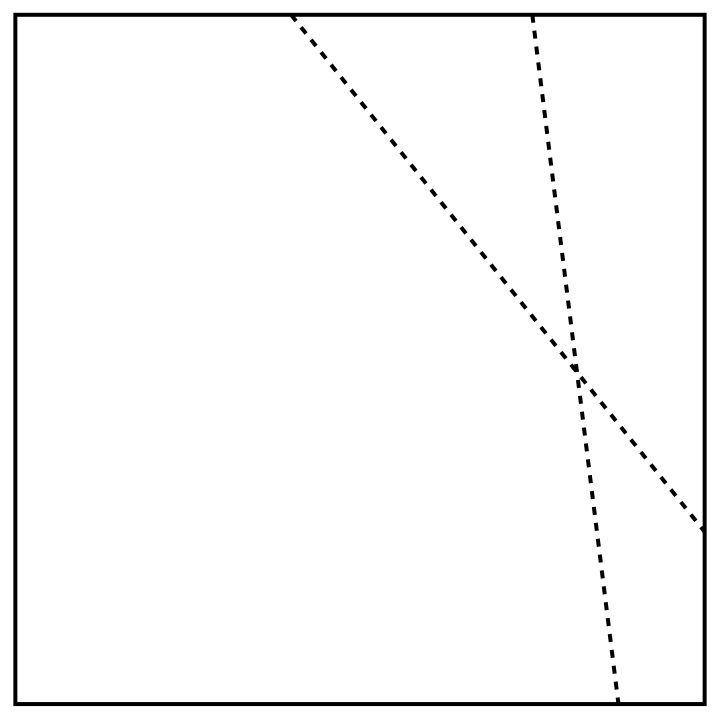}
    \end{minipage}
    \caption{Example taking $A=\begin{pmatrix}0.8 & 0.1\\ 0.5 & 0.4\end{pmatrix}$ and $b={\begin{pmatrix}0.3\\0.4\end{pmatrix}}$. The left-hand side shows the image of $f_b$, while the right-hand side shows its discontinuity set.}
    \label{fig}
\end{figure}



The \textit{code} of $f_b$ is the function $e_b\colon [0,1)^d\to\Zz^d$ defined by
\begin{equation}\label{e}
e_b(x):=\lfloor Ax+b\rfloor, \quad \forall\,x\in[0,1)^d.
\end{equation}

The following properties are elementary to check. 
\begin{lemma}\label{lem:elementary} For every $b\in\Rr^d$,
\begin{enumerate}
\item $f_{b+p}=f_b$ for every $p\in\mathbb{Z}^d$,
\item $f_b$ is injective,
\item $e_{\{b\}}([0,1)^d) \subset \{-1,0,1\}^d$.
\end{enumerate}
\end{lemma}
\begin{proof}
We only give a proof of (2) and (3) as (1) is obvious. 
Let $x,y\in [0,1)^d$ {be} such that $f(x)=f(y)$. By \eqref{f}, $Ax+b-e_b(x)=Ay+b-e_b(y)$, which gives $A(y-x)=e_b(y)-e_b(x)\in\Zz^d$. But,
$
\|e_b(y)-e_b(x)\| = \|A(y-x)\|\leq\|A\|\|y-x\|<1,
$ for every $x,y\in[0,1)^d$. Thus, $e_b(y)=e_b(x)$, which gives $x=y$. This shows (2). Regarding (3), since $\|Ax\|<1$ for every $x\in[0,1)^d$, we have $Ax+\{b\}\in(-1,2)^d$. Thus, $\lfloor Ax + \{b\}\rfloor\subset \{-1,0,1\}^d$, which shows (3).
\end{proof}

{Set}
$$
E_b(p):=e_b^{-1}(p) = \{x\in[0,1)^d\colon e_b(x)=p\}, \quad {p \in\Zz^d}.
$$
Clearly, these sets are disjoint and {for each $p=(p_j)\in\mathbb{Z}^d$,}
$$
E_b(p)=\bigcap_{{j=1}}^d E^{{(j)}}_b(p),\quad E^{{(j)}}_b(p):=\{x\in[0,1)^d\colon {p_j}\leq (Ax+b)_{{j}}<{p_j}+1\}.
$$
Whenever $E_b(p)\neq\emptyset$, we have $f(x) = Ax+b-p$ for every $x\in E_p(b)$. The regions $E_b(p)$, ${p\in\Zz^d},$ are called the \textit{domains of continuity} of $f_b$.

Notice that, by (3) of Lemma~\ref{lem:elementary}, $E_{\{b\}}(p)=\emptyset$ for all $p\notin\{-1,0,1\}^d$. In particular, this shows that $\#\{p\in\mathbb{Z}^d\colon E_b(p)\neq \emptyset\}\leq 3^d$. However, as we shall see in the following lemma, the number of non-empty regions $E_b(p)$ is much smaller. For instance, in the example of Figure~\ref{fig}, $E_b(p)\neq\emptyset$ iff  $p\in\{0,1\}^2$.   

Let $\chi\colon \Rr^d\to \mathbb{Z}^d$ be the function
\begin{equation}\label{eq:qb}
\chi(b) := \min_{x\in[0,1)^d} e_b(x) = \min_{x\in[0,1)^d} \lfloor Ax+b\rfloor,
\end{equation}
where the minimum is taken component-wise. If $A$ has non-negative entries, then $\chi(b)=\lfloor b\rfloor$.
\begin{lemma}\label{lem:chi}
\hfill
\begin{enumerate}
\item $\chi(b)=\chi(\{b\})+\lfloor b\rfloor$ for every $b\in\Rr^d$,
\item $\chi(b)\in\{-1,0\}^d$ for every $b\in[0,1)^d$,
\item $[0,1)^d = \bigcup_{p\in\{0,1\}^d} E_b(\chi(b)+p)$ for every $b\in\Rr^d$,
\item For every $b\in\Rr^d$, $p\in\{0,1\}^d$ and ${j\in\{1,\ldots,n\}}$,
$$
E_b^{{(j)}}(\chi(b)+p)=
\begin{cases}
\{x\in[0,1)^d\colon  (Ax+b)_{{j}}<(\chi(b))_{{j}}+1\},& {\textrm{if}}\quad p_j=0\\
\{x\in[0,1)^d\colon  (Ax+b)_{{j}}\geq (\chi(b))_{{j}}+1\},& {\textrm{if}}\quad p_j=1
\end{cases}
$$
\end{enumerate}
\end{lemma}

\begin{proof}Properties (1) and (2) are obvious. Let us prove (3) and (4).
For every $x,y\in[0,1]^d$, $|(Ax+b)_{{j}}-(Ay+b)_{{j}}|\leq \|A\|\|x-y\|\leq \|A\|<1$. This implies that, for every $b\in\Rr^d$, the set $\{(Ax+b)_{{j}}\colon x\in[0,1)^d \}$ contains at most one integer. Thus,  by \eqref{eq:qb}, we conclude that $(\chi(b))_{{j}} \leq (Ax+b)_{{j}} < (\chi(b))_{{j}} +2$ for every $x\in[0,1)^d$. This shows (4) and also shows (3) since, for every $x\in[0,1)^d$, we have $x\in E_b(\chi(b)+p)$ for some $p\in\{0,1\}^d$.
\end{proof}

\begin{remark}
Lemma~\ref{lem:chi} shows that the discontinuities of $f_b$ partition $[0,1)^d$ into at most $2^d$ regions, i.e. $\#\{p\in\mathbb{Z}^d\colon E_b(p)\neq \emptyset\}\leq 2^d$. Notice that, we are not claiming that all $E_b(\chi(b)+p)$ are non-empty. In fact, it is easy to construct examples of $A$ and $b$ where $f_b$ has a single domain of continuity.
\end{remark}

\subsection{Change of variables}\label{covar}

In order to prove Theorem~\ref{thm1}, it is convenient to change variables using the linear translation $h_b\colon \Rr^d\to\Rr^d$ defined by
\begin{equation}\label{hby} 
h_b(y) := y + (I-A)^{-1}b,\quad y\in\mathbb{R}^d.
\end{equation}

Let $D_b := h_b^{-1}([0,1)^d)$ denote the translated {unit hypercube}. The map $f_b$ in the new variables $y=h_b^{-1}(x)$ is $g_b\colon D_b\to D_b$ defined by $g_b:=h_b^{-1}\circ f_b\circ h_b$. Obviously, $f_b$ is asymptotically periodic iff $g_b$ is {as well}.
\begin{lemma}\label{lem:changevar}
$g_b(y) = Ay - e_b(h_b(y))$ for every $y\in D_b$.
\end{lemma}
\begin{proof}
Clearly, for each $y\in D_b$, 
\begin{align*}
g_b(y)&=f_b(h_b(y))-(I-A)^{-1}b\\
&= Ah_b(y)+b-\lfloor A h_b(y)+b\rfloor - (I-A)^{-1}b\\
&=Ay+A(I-A)^{-1}b+b-(I-A)^{-1}b - \lfloor A h_b(y)+b\rfloor\\
&=Ay - \lfloor A h_b(y)+b\rfloor,
\end{align*}
since $A(I-A)^{-1}b+b-(I-A)^{-1}b=(A+(I-A)-I)(I-A)^{-1}b=0$.
\end{proof}
The domains of continuity of $g_b$ are 
\begin{equation}\label{eq:Cbp}
C_b^p:=h_b^{-1}(E_b(\chi(b)+p)),\quad p\in\{0,1\}^d.
\end{equation}
Whenever $C_b^p\neq\emptyset$, we have $g_b(y)=Ay-(\chi(b)+p)$ for every $y\in C_b^p$.
\subsection{Change of parameters}
 In order to describe the domains $C_b^p$ as in the setup of {\cite{GaivaoPires2024}}, we introduce the following change of parameters.
\begin{defn}
Let $\rho\colon \Rr^d\to \Rr^d$ be the map $b\mapsto \mu=(\mu_1,\ldots,\mu_d)$ where
$$
\mu_{{j}} = \frac{(\chi(b)-(I-A)^{-1}b)_{{j}} +1}{\|A^{({{j}})}\|},\quad {1\le j\le d},
$$
where $A^{{(j)}}$ denotes the ${j}$-th row of $A$.
\end{defn}

Given $\mu\in\Rr^d$, consider the following family of hyperplanes,
\begin{equation}\label{H} 
H_{{j}}(\mu) := \left\{ x\in\mathbb{R}^d: \langle v^{({j})}, x\rangle = \mu_{{j}}\right\},\quad 1\leq {j} \leq d
\end{equation}
where $\langle\cdot,\cdot\rangle$ denotes the Euclidean inner product and $v^{{(j)}} := \frac{A^{{(j)}}}{\|A^{{(j)}}\|}$ are unit vectors. Associated to this family of hyperplanes we have 
the \textit{label map} $\sigma_\mu\colon \Rr^d\to\mathcal{A}$ defined by $\sigma_\mu(x)=(s_1(x),\ldots,s_d(x))$ where
$$
s_{{j}}(x)=\begin{cases}+1,& \text{if } \langle v^{{(j)}},x\rangle < \mu_{{j}}\\
-1,& \text{if } \langle v^{{(j)}},x\rangle \geq \mu_{{j}}
\end{cases},
$$
and $\mathcal{A}=\{-1,+1\}^d$.
For simplicity, set $(-1)^p = ((-1)^{p_1},\ldots,(-1)^{p_d})\in\mathcal{A}$ for any $p={(p_j)}\in\{0,1\}^d$.
\begin{lemma}\label{lem:Cbp}For every $b\in\Rr^d$ and $p\in\{0,1\}^d$,
$$C_b^p = D_b\cap \sigma^{-1}_{\rho(b)}((-1)^p).$$
\end{lemma}
\begin{proof}
Suppose that ${p_j}=0$ for some ${1\le j\le d}$. By \eqref{eq:Cbp} and by (4) of Lemma~\ref{lem:chi} we get,
\begin{align*}
h_b^{-1}(E_b^{{(j)}}(\chi(b)+p))&=\{h_b^{-1}(x)\colon x\in[0,1)^d,\, (Ax+b)_{{j}}<(\chi(b))_{{j}}+1\}\\
&= \{y\in D_b\colon (Ah_b(y)+b)_{{j}}<(\chi(b))_{{j}}+1\}\\
&= \{y\in D_b\colon (Ay+(A(I-A)^{-1}+I)b)_{{j}}<(\chi(b))_{{j}}+1\}\\
&= \{y\in D_b\colon (Ay+(I-A)^{-1}b)_{{j}}<(\chi(b))_{{j}}+1\}\\
&= \{y\in D_b\colon \langle A^{{(j)}},y\rangle <(\chi(b)-(I-A)^{-1}b)_{{j}}+1\}\\
&= \{y\in D_b\colon \langle v^{{(j)}},y\rangle <\mu_{{j}}\}\\
&= D_b\cap s_{{j}}^{-1}((-1)^{p_{{j}}}).
\end{align*}
Similarly, for $p_{{j}}=1$, we get $h_b^{-1}(E_b^{{(j)}}(\chi(b)+p)) = D_b\cap s_{{j}}^{-1}((-1)^{p_{{j}}})$.

\end{proof}

The next step is to extend $g_b$ to a piecewise-affine contraction map defined on the whole $\Rr^d$. The extension is the obvious one. 

\subsection{Extension map $G_\mu$.}

Throughout this subsection we fix $k\in \Zz^d$. Recall that $\mathcal{A}=\{-1,+1\}^d$.

Let $\{\varphi_{{i}}\colon \Rr^d\to\Rr^d\}_{{i\in\mathcal{A}}}$ be the family of affine contractions $$\varphi_{{i}}(y)=Ay-(k+p), \quad{\forall y\in\mathbb{R}^d},$$ where {$p\in\{0,1\}^d$ such that} $(-1)^p={i}$.

Define
$$
r:=\frac{(1+\|A\|)(1+\|k\|)}{(1-\|A\|)^2}.
$$
\begin{lemma}\label{lem:invariantball} $\|y\|\leq 2r \Longrightarrow \|\varphi_{{i}}(y)\|< 2r$ {for all $i\in\mathcal{A}$}.
\end{lemma}
\begin{proof}
Each $\varphi_{{i}}$ has a unique fixed point $z_{{i}} = -(I-A)^{-1}(k+p)$. Notice that $\|z_{{i}}\|\leq \frac{1+\|k\|}{1-\|A\|}$ for all ${i}\in\mathcal{A}$. Then, given $\|y\|\leq {2r}$, we have
\begin{align*}
\|\varphi_{{i}}(y)\| &= \|\varphi_{{i}}(y)- \varphi_{{i}}(z_{{i}}) + z_{{i}}\|\\
&\leq\|A(y-z_{{i}})\| + \|z_{{i}}\|\\
&\leq 2\|A\| r + (1+\|A\|)\|z_{{i}}\|\\
&\leq 2\|A\| r + \frac{(1+\|A\|)(1+\|k\|)}{1-\|A\|}\\
&= 2\|A\| r + (1-\|A\|)r\\
&= (1+\|A\|)r\\
&< 2r.
\end{align*}
\end{proof}
Let 
$$
X=\{y\in\Rr^d\colon \|y\|\leq 2 r\}.
$$
By Lemma~\ref{lem:invariantball}, $\varphi_{{i}}(X)\subset \mathrm{int}(X)$. Therefore, 
given $\mu\in\Rr^d$, we can define $G_\mu\colon X\to X$ by 
$$
G_\mu(y)=\varphi_{\sigma_\mu(y)}(y),\quad y\in \Rr^d.
$$
The map $G_\mu$ is a piecewise-affine contraction as defined in \cite[Equation (2.9)]{GaivaoPires2024}}. We will now recall a result from \cite{GaivaoPires2024}, but first have to introduce the {notions} of regular point {and asymptotic periodicity on a forward invariant set}. 
\begin{defn}[{regular point}]\label{def:regular} {Given $\mu\in\mathbb{R}^d$,}
a point $y\in X$ is called \textit{a regular point of} $G_\mu$ if 
$$
G_\mu^n(y) \in \bigcup_{\alpha\in\mathcal{A}}\mathrm{int}(\sigma_\mu^{-1}(\alpha)),\quad \forall\,n\geq0.
$$
The set of regular points of $G_\mu$ is denoted by $Z_\mu$.
\end{defn}
{It is an elementary  fact that $Z_\mu$ is forward invariant by $G_\mu$.}
{
\begin{defn}[asymptotic periodicity on a set] Let $\mu\in\mathbb{R}^d$ and
 $Z\subset X$ be a set forward invariant by $G_\mu$. We say that $G_\mu$ is asymptotically periodic on $Z$ if $\bigcup_{z\in Z} \omega(G_\mu,z)$ is the union of finitely many periodic orbits of $G_\mu$. 
\end{defn}
}
{
Notice that if $G_{\mu}$ is asymptotically periodic on $Z=X$, then $G_{\mu}$ is called asymptotically periodic, as it has been defined in the Introduction.}

\begin{theorem}[Theorem 2.12 of \cite{GaivaoPires2024}]\label{thm:Gmu}
The map $G_\mu$ is asymptotically periodic on $Z_\mu$ for Lebesgue almost every $\mu\in\Rr^d$.
\end{theorem}

We generalise the arguments used in the proof of  \cite[Corollary 2.13]{GaivaoPires2024} to obtain the following stronger result.

\begin{theorem}\label{thm:Gmu2}
The map $G_\mu$ is asymptotically periodic for Lebesgue almost every $\mu\in\Rr^d$.
\end{theorem}

\begin{proof}
By Theorem~\ref{thm:Gmu}, there is a full Lebesgue measure set $W\subset\Rr^d$ such that the map $G_\mu$ is asymptotically periodic on $Z_\mu$ for every $\mu\in W$. So, {since $Z_\mu$ is forward invariant by $G_\mu$}, it remains to show that for every $y\in X$, there is $n\geq0$ such that $G_\mu^n(y)\in Z_\mu$. To prove {this}, we need to introduce some more terminology.

Given $\alpha= (i_0, i_1,\ldots, i_{n-1})\in\mathcal{A}^n$, $n\ge 1$,  the map $\varphi^{\alpha}:=\varphi_{i_{n-1}}\circ\cdots\circ\varphi_{i_0}$ satisfies
$$ \varphi^{\alpha}(y) = A^{n} y + r_{\alpha},\quad\forall y\in\mathbb{R}^d,
$$
where $r_\alpha =\varphi^\alpha(0)$. Given $\mu\in\Rr^d$ and $n\geq1$, define the collection of hyperplanes
$$
\mathcal{H}_\mu^{(n)}:=\{(\varphi^\alpha)^{-1}(H_j(\mu))\colon \alpha\in\mathcal{A}^k,\,0\leq k<n,\,1\leq j\leq d\},
$$
where $\varphi^0=\mathrm{id}$ is the identity map and $\mathcal{A}^0 = \{0\}$. By \cite[Lemma 5.8]{GaivaoPires2024},  
\begin{equation}\label{eq:bound}
\forall\,\mu\in W,\quad \sup_{n\geq1}\max_{x\in\Rr^d}\#\{H\in\mathcal{H}_\mu^{(n)}\colon x\in H\}<+\infty.
\end{equation}
Let
$$ \Lambda = \bigcup_{n\in\mathbb{N}^*} \{\lambda\in\mathbb{R}: \lambda \textrm{  is {a real} eingenvalue of $A^{n}$}\}.
$$
Because $\Vert A\Vert<1$, we have that
$\Lambda \subset (-1,1)$.
Given integers $1\le m<n$, $\alpha\in\mathcal{A}^{m}$ and $\beta\in\mathcal{A}^{n},$ 
let
$$
N_{\alpha,\beta} = \left\{\mu\in\Rr^d\colon \exists\,j\in\{1,\ldots,d\},\, \exists\,\lambda\in\Lambda \mid \mu_j = \frac{1}{1-\lambda}\langle r_\alpha-\lambda r_\beta,v^{(j)}\rangle \right\},
$$
where $N_{\alpha,\beta}=\emptyset$ if $A$ has no real eigenvalues. The set $N_{\alpha,\beta}$ is a union of finitely many hyperplanes. Thus, $N_{\alpha,\beta}$ has zero Lebesgue measure. 
Set 
$$
W_1 = W{\bigg\backslash}\left( {\bigcup_{m\geq1}\bigcup_{n>m}\bigcup_{\alpha\in\mathcal{A}^m} 
\bigcup_{\beta\in\mathcal{A}^n}
N_{\alpha,\beta}}\right).$$
Then, $W_1$ has full Lebesgue measure.  

Now, the rest of the proof goes by contradiction.  Let $\mu\in W_1$ and suppose that there is  $y\in X$ such that $G_\mu^n(y)\notin Z_\mu$ for all $n\geq0$.  Since $G_\mu(X)\subset \mathrm{int}(X)$,  the orbit ${\{G^n_\mu(y)\}}_{n\geq0}$ is contained in the union of hyperplanes $\bigcup_{j=1}^d H_j(\mu)$. By the pigeon-hole principle, there must exist $j\in\{1,\ldots,d\}$ such that ${\{G^n_\mu(y)\}}_{n\geq0}$ intersects the hyperplane $H_j(\mu)$ infinitely many times, i.e.,  $G^n_{\mu}(y)\in H_j(\mu)$ for infinitely many $n\geq1$.  Equivalently,  there exist infinitely many $n\ge 1$ such that $y\in (\varphi^{\alpha})^{-1}(H_j(\mu))$ for some $\alpha\in\mathcal{A}^n$.  By \eqref{eq:bound},  there exist positive integers $m<n$, $\alpha\in\mathcal{A}^{m}$ and $\beta\in\mathcal{A}^{n}$ such that $y$ belongs to the coincident hyperplanes 
\begin{equation}\label{151}
(\varphi^{\beta})^{-1}(H_j(\mu)) = (\varphi^{\alpha})^{-1}(H_j(\mu)).
\end{equation}
In terms of vectors, we have that for some $\lambda\neq 0$,
$$
(A^{n})^t v^{(j)} = \lambda (A^{m})^t v^{(j)}.
$$
Since $A$ is invertible, we have that 
$ (A^{n-m})^t v^{(j)} = \lambda v^{(j)}$, thus $\lambda\in \Lambda$.
Moreover,  because $y$ belongs to both hyperplanes in \eqref{151},  we have that
\begin{align*}
\mu_j &= \langle A^n y+r_\beta,v^{(j)}\rangle\\
&= \langle A^n y,v^{(j)}\rangle +\langle r_\beta,v^{(j)}\rangle\\
&= \langle y,(A^n)^t v^{(j)}\rangle +\langle r_\beta,v^{(j)}\rangle\\
&= \langle y,\lambda (A^{m})^t v^{(j)}\rangle +\langle r_\beta,v^{(j)}\rangle\\
&= \lambda \langle A^{m} y,  v^{(j)}\rangle +\langle r_\beta,v^{(j)}\rangle\\
&= \lambda (\mu_j-\langle r_\alpha, v^{(j)}\rangle) +\langle r_\beta,v^{(j)}\rangle,
\end{align*}
which gives $(1-\lambda)\mu_j= \langle r_\beta - \lambda r_\alpha,v^{(j)}\rangle$. Thus,  $\mu \in N_{\alpha,\beta}$,  yielding a contraction.  
\end{proof}

\subsection{Proof of Theorem~\ref{thm1}}

By (1) of Lemma~\ref{lem:elementary}, we may suppose that $b\in[0,1)^d$. 
By (2) of Lemma~\ref{lem:chi}, the function $\chi$ is piecewise constant in $[0,1)^d$ taking values in $\{-1,0\}^d$. Therefore, given any $k\in\{-1,0\}^d$, let $U_k\subset \Rr^d$ be the pre-image of $k$ under $\chi$, i.e. $U_k:= \chi^{-1}(k)$. In order to prove Theorem~\ref{thm1}, since $[0,1)^d=\bigcup_{k\in\{-1,0\}^d} U_k$, it is enough to show that $f_b$ is asymptotically periodic for Lebesgue almost every $b\in U_k$. Changing variables, we just need to show that  $g_b$ is asymptotically periodic for Lebesgue almost every $b\in U_k$ (see {Section \ref{covar}}).

Now, we show that $G_\mu$ is an extension of $g_b$.
\begin{lemma}\label{lem:extension}
Let $b\in[0,1)^d$. Then $D_b\subset X$ and $G_\mu|_{D_b} = g_b$ provided $k=\chi(b)$ and $\mu=\rho(b)$.
\end{lemma}

\begin{proof}
First, we show the inclusion $D_b\subset X$. Recall the definition of $D_b= h_b^{-1}([0,1)^d)$. Given any $y\in D_b$ we have $y= x- (I-A)^{-1}b$ for some $x\in[0,1)^d$. This gives $\|y\|\leq 1+\frac{\|b\|}{1-\|A\|}\leq \frac{2-\|A\|}{1-\|A\|}$. But, a simple computation shows that 
$$
\frac{2-\|A\|}{1-\|A\|}< 2\frac{(1+\|A\|)(1+\|k\|)}{(1-\|A\|)^2},\quad \forall\, k\in \Zz^d.
$$
Thus, $\|y\|<2r$ and $D_b\subset X$. 

Now, we show that $G_\mu$ is an extension of $g_b$. Let $y\in D_b$. Since $\{C_b^p\}_{p\in\{0,1\}^d}$ is a partition of $D_b$, there must exist $p\in\{0,1\}^d$ such that $y\in C_b^p$. for some $p\in\{0,1\}^d$. Then, by Lemma~\ref{lem:Cbp}, $\sigma_{\rho(b)}(y)=(-1)^p$, which gives 
$$
G_{\rho(b)}(y)=\varphi_{\sigma_{\rho(b)}(y)}(y)= \varphi_{(-1)^p}(y) = Ay-(\chi(b) + p)= g_b(y),
$$
where the last equality follows from Lemma~\ref{lem:changevar}.
\end{proof}

Using this lemma, for any $b\in U_k$, we have $k=\chi(b)$ and $g_b$ equals the restriction of $G_{\rho(b)}$ on $D_b$. But, by Theorem~\ref{thm:Gmu2}, there is a set $W\subset \Rr^d$ of full Lebesgue measure such that $G_\mu$ is asymptotically periodic for every $\mu\in W$. Therefore, since $(\rho(b))_{{j}}=\frac{(k-(I-A)^{-1}b)_{{j}} + 1}{\|A^{{(j)}}\|}$ for $b\in U_k$, we see that $\rho|_{U_k}$ is an injective affine transformation. Thus, $(\rho|_{U_k})^{-1}(W)$ has full Lebesgue measure in $U_k$ and we have proven Theorem~\ref{thm1}.\qed

\subsection{Proof of Theorem~\ref{theoremtwo}} Let $\theta$ be an irrational number and $a>0$ be so small that the matrix $$A:=a\begin{bmatrix} \cos2\pi\theta & -\sin2\pi\theta\\ \sin2\pi\theta & \phantom{-}\cos2\pi\theta
\end{bmatrix}$$ 
satisfies $A\in\mathrm{GL}_2$ and $\|A\|<1$. By Theorem \ref{thm1},
there exists a full Lebesgue set $U\subset\mathbb{R}^2$ such that $f$, or more precisely, the map $f_b$ defined by \eqref{mapf2}, is asymptotically periodic. This implies that $f_b$ has at least one regular periodic orbit, that is, a periodic orbit that passes through no discontinuity of $f_b$. In particular, there exist an integer $k\ge 1$ and a $k$-periodic point $x^*\in (0,1)^2$ such that $f_b^k$ is smooth around $x^*$ and $f_b^k(x^*)=x^*$. Suppose that there exist a $C^1$ diffeomorphsim $h: [0,1)^2\to [0,1)^2$ and one-dimensional contracted rotations $g_1,g_2:[0,1)\to [0,1)$ such that $$\left(h\circ f_b\circ h^{-1}\right)(y_1,y_2)= (g_1\times g_2)(y_1,y_2)= \left(g_1(y_1),g_2(y_2)\right) \quad\textrm{for all}\quad  (y_1,y_2)\in [0,1)^2.$$
Let $(y_1^*,y_2^*)=y^*:=h(x^*)$. Then $g_1\times g_2$ is smooth around $y^*$ and $(g_1\times g_2)^k(y^*)=y^*$. 
We conclude that
\begin{equation}\label{90}
\left(h\circ f_b^k\circ h^{-1}\right)(y)= (g_1^k\times g_2^k)(y)\quad\textrm{for all}\quad  y\in [0,1)^2.
\end{equation}
Moreover, $h\circ f_b^k\circ h^{-1}$ and $g_1^k\times g_2^k$ are both $C^1$ around $y^*$. Derivating \eqref{90} at $y=y^*$ yields 
$$
Dh (x^*) Df_b^k(x^*)Dh^{-1}(y^*) = D(g_1^k\times g_2^k)(y^*),
$$ that is,
$$Dh (x^*) Df_b^k(x^*)\left[Dh(x^*)\right]^{-1} = D(g_1^k\times g_2^k)(y^*)=\begin{bmatrix} Dg_1^k(y_1^*)& 0 \\ 0 & Dg_2^k(y_2^*)\end{bmatrix}.$$ Therefore, $Df_b^k(x^*)$ is similar to a diagonal matrix implying that $Df_b^k(x^*)$ has only real eigenvalues. This is a contradiction because $Df_b^k(x^*)=A^k$ and $A^n$ has no real eigenvalue for any $n\geq1$.

\appendix
\section{Bifurcations of the rotation number}\label{appendix figure}

\begin{figure}[htbp]
    \centering
    \includegraphics[width=0.8\textwidth]{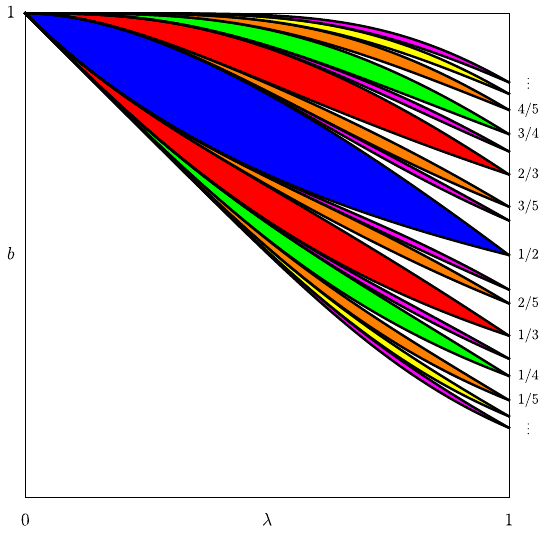}
    \caption{The relation between the parameteres $\lambda$, $b$ and the rotation number of $f_{\lambda,b}$.}
    \label{fig2}
\end{figure}

%


\bibliographystyle{amsplain}
\bibliography{Bibfile.bib}

@article{GaivaoPires2024,
  AUTHOR = {Gaiv\~ao, Jos\'e Pedro and Pires, Benito},
  TITLE = {Multi-dimensional piecewise contractions are asymptotically periodic},
  PAGES = {983--1008},
  JOURNAL = {Ergodic Theory and Dynamical Systems},
 VOLUME = {46},
  YEAR = {2026},
  eprint       = {Published online 2025},
  DOI = {10.1017/etds.2025.10250},
}

@Article{zbMATH01300448,
 Author = {Bugeaud, Yann and Conze, Jean-Pierre},
 Title = {Computing the dynamics of contracting linear transformations mod 1 and {Farey} arcs},
 FJournal = {Acta Arithmetica},
 Journal = {Acta Arith.},
 ISSN = {0065-1036},
 Volume = {88},
 Number = {3},
 Pages = {201--218},
 Year = {1999},
 Language = {French},
 DOI = {10.4064/aa-88-3-201-218},
 zbMATH = {1300448},
 Zbl = {0935.11008}
}

@article {CGMU16,
    AUTHOR = {Catsigeras, E. and Guiraud, P. and Meyroneinc, A. and Ugalde,
              E.},
     TITLE = {On the asymptotic properties of piecewise contracting maps},
   JOURNAL = {Dyn. Syst.},
  FJOURNAL = {Dynamical Systems. An International Journal},
    VOLUME = {31},
      YEAR = {2016},
    NUMBER = {2},
     PAGES = {107--135},
      ISSN = {1468-9367,1468-9375},
   MRCLASS = {37B25 (37B20 37B40 37N99)},
  MRNUMBER = {3494598},
MRREVIEWER = {Lori\ Alvin},
       DOI = {10.1080/14689367.2015.1068274},
       URL = {https://doi.org/10.1080/14689367.2015.1068274},
}

@article {B93,
    AUTHOR = {Bugeaud, Yann},
     TITLE = {Dynamique de certaines applications contractantes,
              lin\'{e}aires par morceaux, sur {$[0,1)$}},
   JOURNAL = {C. R. Acad. Sci. Paris S\'{e}r. I Math.},
  FJOURNAL = {Comptes Rendus de l'Acad\'{e}mie des Sciences. S\'{e}rie I.
              Math\'{e}matique},
    VOLUME = {317},
      YEAR = {1993},
    NUMBER = {6},
     PAGES = {575--578},
      ISSN = {0764-4442},
   MRCLASS = {94A12},
  MRNUMBER = {1240802},
}

@article {BD08,
    AUTHOR = {Bruin, Henk and Deane,  Jonathan H.  B.},
     TITLE = {Piecewise Contractions Are Asymptotically Periodic},
   JOURNAL = {Proceedings of the American Mathematical Society},
    VOLUME = {137},
      YEAR = {2008},
    NUMBER = {4},
     PAGES = {1389--1395},
}

@article {LN18,
    AUTHOR = {Laurent, Michel and Nogueira, Arnaldo},
     TITLE = {Rotation number of contracted rotations},
   JOURNAL = {J. Mod. Dyn.},
  FJOURNAL = {Journal of Modern Dynamics},
    VOLUME = {12},
      YEAR = {2018},
     PAGES = {175--191},
      ISSN = {1930-5311,1930-532X},
   MRCLASS = {11J91 (37E45)},
  MRNUMBER = {3815128},
MRREVIEWER = {Tapas\ Chatterjee},
       DOI = {10.3934/jmd.2018007},
       URL = {https://doi.org/10.3934/jmd.2018007},
}

@article{JAIN_LIVERANI_2025, 
title={Piecewise contractions}, 
volume={45}, 
DOI={10.1017/etds.2024.78}, 
number={5}, 
journal={Ergodic Theory and Dynamical Systems}, 
author={Jain, Sakshi and Liverani, Carlangelo}, 
year={2025}, 
pages={1503–1540}}

@article{zbMATH04182365,
 author = {Gambaudo, Jean-Marc and Tresser, Charles},
 title = {On the dynamics of quasi-contractions},
 fjournal = {Boletim da Sociedade Brasileira de Matem{\'a}tica},
 journal = {Bol. Soc. Bras. Mat.},
 issn = {0100-3569},
 volume = {19},
 number = {1},
 pages = {61--114},
 year = {1988},
 language = {English},
 doi = {10.1007/BF02584821},
 zbMATH = {4182365},
 Zbl = {0717.54022}
}

@article {BKLN21,
    AUTHOR = {Bugeaud, Yann and Kim, Dong Han and Laurent, Michel and
              Nogueira, Arnaldo},
     TITLE = {On the {D}iophantine nature of the elements of {C}antor sets
              arising in the dynamics of contracted rotations},
   JOURNAL = {Ann. Sc. Norm. Super. Pisa Cl. Sci. (5)},
  FJOURNAL = {Annali della Scuola Normale Superiore di Pisa. Classe di
              Scienze. Serie V},
    VOLUME = {22},
      YEAR = {2021},
    NUMBER = {4},
     PAGES = {1691--1704},
      ISSN = {0391-173X,2036-2145},
   MRCLASS = {11J91 (37A44 37E05)},
  MRNUMBER = {4360601},
MRREVIEWER = {Wolfgang\ Steiner},
}

@article {B04,
    AUTHOR = {Bugeaud, Yann},
     TITLE = {Linear mod one transformations and the distribution of
              fractional parts {$\{\xi(p/q)^n\}$}},
   JOURNAL = {Acta Arith.},
  FJOURNAL = {Acta Arithmetica},
    VOLUME = {114},
      YEAR = {2004},
    NUMBER = {4},
     PAGES = {301--311},
      ISSN = {0065-1036,1730-6264},
   MRCLASS = {11K31 (37A45)},
  MRNUMBER = {2101819},
MRREVIEWER = {Gerhard\ Larcher},
       DOI = {10.4064/aa114-4-1},
       URL = {https://doi.org/10.4064/aa114-4-1},
}

@article{Arnaud,
  author   = {Coutinho, R. and Fernandez, B. and Lima, R. and Meyroneinc, A.},
  title    = {Discrete time piecewise affine models of genetic regulatory networks},
  journal  = {J. Math. Biol.},
  volume   = {52},
  year     = {2006},
  number   = {4},
  pages    = {524--570},
  doi      = {10.1007/s00285-005-0359-x}
}

@incollection {MR1861988,
    AUTHOR = {Bugeaud, Yann and Conze, Jean-Pierre},
     TITLE = {Dynamics of some contracting linear functions modulo 1},
 BOOKTITLE = {Noise, oscillators and algebraic randomness ({C}hapelle des
              {B}ois, 1999)},
    SERIES = {Lecture Notes in Phys.},
    VOLUME = {550},
     PAGES = {379--387},
 PUBLISHER = {Springer, Berlin},
      YEAR = {2000},
      ISBN = {3-540-67572-8},
   MRCLASS = {37E05 (11K50 37N99 39B12 94A12)},
  MRNUMBER = {1861988},
       DOI = {10.1007/3-540-45463-2\_20},
       URL = {https://doi.org/10.1007/3-540-45463-2_20},
}

\end{document}